\documentclass[12pt]{amsart}
\usepackage{amssymb}
\usepackage{amsmath}
\usepackage{amssymb}
\usepackage{amsthm}
\usepackage{enumerate}
\usepackage[all,arc,curve,color,frame]{xy}

\newcommand{\C}{\ensuremath{\mathbb{C}}}

\newcommand{\p}{\partial}

\newcommand{\Z}{\ensuremath{\mathbb{Z}}}
\newtheorem{lemma}{Lemma}
\newtheorem{proposition}{Proposition}
\newtheorem{theorem}{Theorem}
\newtheorem{corollary}{Corollary}
\begin{document}

\title[The phase form of a deformation quantization]
{On the phase form of a deformation quantization with separation of variables}
\author[Alexander Karabegov]{Alexander Karabegov}
\address[Alexander Karabegov]{Department of Mathematics, Abilene
Christian University, ACU Box 28012, Abilene, TX 79699-8012}
\email{axk02d@acu.edu}

\begin{abstract}
Given a star product with separation of variables on a pseudo-K\"ahler manifold, we obtain a new formal (1,1)-form from its classifying form and call it the phase form of the star product. The cohomology class of a star product with separation of variables equals the class of its phase form. We show that the phase forms can be arbitrary and they bijectively parametrize the star products with separation of variables. We also describe the action of a change of the formal parameter on  a star product with separation of variables, its formal Berezin transform, classifying form, phase form, and canonical trace density.
\end{abstract}
\date{March 16, 2015}
\maketitle

\section{Introduction}

Given a manifold $M,$ we denote by $C^\infty(M)[\nu^{-1},\nu]]$ the space of formal Laurent series with a finite principal part,
\[
     f = \sum_{r = k}^\infty \nu^r f_r,
\]
where $f_r \in C^\infty(M)$ and $k \in \Z$. If $k > 0$, we will say that the formal function $f$ vanishes at $\nu=0$.

Deformation quantization on  a Poisson manifold $(M, \{\cdot,\cdot\})$  is an associative product (named a star product) $\ast$ on $C^\infty(M)[\nu^{-1},\nu]]$ given by the following $\nu$-adically convergent series,
\[
     f \ast g = fg + \sum_{r = 1}^\infty \nu^r C_r(f,g),
\]
where $C_r$ are bidifferential operators and $C_1(f,g) - C_1(g,f) = i\{f,g\}$. We assume that the unit constant is the unity for the star product, $f \ast 1 = 1 \ast f = f$.  A star product can be restricted (localized) to an open subset $U \subset M$. We denote by $L_f$ and $R_f$ the operators of left and right star multiplication by a function $f$, respectively,  so that $f \ast g = L_f g = R_g f$. The operators $L_f$ and $R_g$ commute for any functions $f, g$.

Nondegenerate Poisson bivectors bijectively correspond to symplectic forms. We will call a star product nondegenerate if the corresponding Poisson structure is nondegenerate.

Two star products $\ast$ and $\ast'$ on a Poisson manifold $M$ are called equivalent if there exists a formal differential operator $B = 1 + \nu B_1 + \ldots$ on $M$ such that
\[
                               f \ast' g = B^{-1}(Bf \ast Bg).
\] 
Deformation quantization was introduced in \cite{BFFLS}. It was proved by a number of authors (\cite{D},\cite{F2},\cite{NT2}) that the equivalence classes of star products on a symplectic manifold $(M, \omega_{-1})$ are bijectively parametrized by the formal cohomology classes in
\[
      \frac{1}{\nu}\left[\omega_{-1}\right] + H^2(M)[[\nu]].
\]
Fedosov gave a simple geometric construction of star products in each equivalence class on an arbitrary symplectic manifold in \cite{F1}. 
Kontsevich proved in \cite{K} that star products exist on arbitrary Poisson manifolds and gave a classification of star products up to equivalence in terms of formal deformations of the Poisson structure.

If $\ast$ is a star product on a $2m$-dimensional symplectic manifold $(M, \omega_{-1})$, there exists a canonically normalized formal trace density
\begin{equation}\label{E:mu}
          \mu = \frac{1}{m!}\left(\frac{1}{\nu}\, \omega_{-1}\right)^m e^{\varkappa},   
\end{equation}
where $\varkappa$  is a globally defined formal function on~ $M$ which vanishes at $\nu =0$ (see \cite{F1}, \cite{LMP2}). On a compact $M$, the index theorem for deformation quantization (\cite{F2},\cite{NT}) gives a topological formula for the
total volume of $\mu$ in terms of the cohomology class of the star product.  

If $M$ is a complex manifold with a  Poisson bracket $\{\cdot,\cdot\}$ of type $(1,1)$ with respect to the complex structure, a star product $\ast$ on $(M, \{\cdot,\cdot\})$ has the property of separation of variables (of the anti-Wick type) if $a \ast f = af$ and $f \ast b = bf$ for any locally defined holomorphic function $a$ and antiholomorphic function $b$, i.e., the operators $L_a = a$ and $R_b = b$ are pointwise multiplication operators. Equivalently, every bidifferential operator $C_r$ of the star product $\ast$ differentiates its first argument in antiholomorphic directions and the second argument in holomorphic ones.

On a coordinate chart on $M$, the Poisson bracket is given by a Poisson tensor $g^{\bar lk}$ of type (1,1),
\[
      \{f,g\} = i g^{\bar lk}\left(\frac{\p f}{\p z^k}\frac{\p g}{\p \bar z^l} - \frac{\p g}{\p z^k}\frac{\p f}{\p \bar z^l}\right).
\]
 For any star product with separation of variables on $(M,\{\cdot,\cdot\})$ the operator $C_1$ is given locally by the formula
\[
               C_1(f,g) = g^{\bar lk} \frac{\p f}{\p \bar z^l} \frac{\p g}{\p z^k}.
\]
Given a star product with separation of variables $\ast$ on $M$, there exists a formal differential operator $I = 1 + \nu I_1 + \nu^2 I_2 + \ldots$ globally defined on $M$ such that for a local holomorphic function $a$ and an antiholomorphic function $b$,
\[
                     I(ab) = b \ast a.
\]
In particular, $Ia = a$ and $Ib = b$ and therefore $I_ra =0$ and $I_rb =0$ for $r \geq 1$.  The operator $I_1$ is the Laplace-Beltrami operator $\Delta$ given by the local formula 
\[
                 \Delta = g^{\bar lk}\frac{\p^2}{\p z^k \p \bar z^l}.
\]

If $\lambda$ is a closed global (1,1)-form on $M$, its local potential $\varphi$ on an open subset $U \subset M$ is defined up to a summand $a + b$ and therefore the function $I_r \varphi$ does not depend on the choice of the potential. Such functions on a contractible covering glue to a global function on $M$ which depends only on the form $\lambda$ and the operator $I_r$. This observation will be used in the proof of Theorem \ref{T:existuniq}.

A star product with separation of variables $\ast$ on a complex manifold $M$ is completely determined by its formal Berezin transform $I$. It was proved in \cite{CMP3} that the inverse operator $I^{-1}$ is the formal Berezin transform of a star product with separation of variables $\tilde\ast$ on $M$ equipped with the opposite Poisson structure such that
\[
                      f \tilde\ast g = I^{-1}(Ig \ast If).
\]
We call $\tilde\ast$ the dual of the star product $\ast$. The dual of $\tilde\ast$ is $\ast$.

Given a complex manifold $M$ which admits a pseudo-K\"ahler structure, we denote by $\Omega(M)$ the set of formal series
\[
    \omega = \frac{1}{\nu}\, \omega_{-1} + \omega_0 + \nu \omega_1 + \ldots,
\]
where $\omega_r, r \geq -1$, are closed (1,1)-forms on $M$ and $\omega_{-1}$ is nondegenerate. In particular, $(M, \omega_{-1})$ is a pseudo-K\"ahler manifold which has a Poisson structure corresponding to $\omega_{-1}$. It was proved in \cite{CMP1} that the nondegenerate star products with separation of variables on $M$ are bijectively parametrized by the elements of $\Omega(M)$. In \cite{BW} Fedosov's geometric construction was modified in order to show that on every pseudo-K\"ahler manifold there exists a star product with separation of variables (of the Wick type). In \cite{N2} it was shown that every star product with separation of variables on a pseudo-K\"ahler manifold can be obtained via a generalized Fedosov's construction.

The form $\omega$ parametrizing a nondegenerate star product with separation of variables ~$\ast$ is called its classifying form. The classifying form $\tilde\omega$ of the dual star product $\tilde\ast$ is such that
\[
    \tilde\omega = -\frac{1}{\nu}\, \omega_{-1} + \tilde\omega_0 + \nu \tilde\omega_1 + \ldots.
\]
The mapping $\omega \mapsto \tilde\omega$ is an involution on $\Omega(M)$.
We call the form
\begin{equation}\label{E:phase}
      \omega^{ph} = \frac{1}{2}\left(\omega - \tilde\omega\right) = \frac{1}{\nu}\, \omega_{-1} + \omega^{ph}_0 + \nu \omega^{ph}_1 + \ldots
\end{equation}
the phase form of the star product $\ast$. Clearly, the phase form of the dual star product $\tilde\ast$ is $-\omega^{ph}$. Given a form $\omega \in \Omega(M)$, we will call the corresponding forms $\tilde\omega$ and $\omega^{ph}$ its dual and phase forms, respectively.

In this paper we will prove that the mapping $\omega \mapsto \omega^{ph}$ is a bijection of $\Omega(M)$ onto itself and thus the phase forms can be used as an alternative parametrization of the nondegenerate star products with separation of variables on $M$. 
This choice of parametrization is justified by the fact that the cohomology class of a star product with separation of variables is equal to the cohomology class of its phase form, which follows from results obtained in ~\cite{LMP1}.

We will consider the action of a change of the formal parameter on a star product with separation of variables, its formal Berezin transform, classifying form, phase form, and canonical trace density. In particular, we will consider the action of proper involutions of the formal parameter such as $\nu \mapsto -\nu$ and describe the star products whose phase form is odd with respect to an involution.

\section{Deformation quantizations with separation of variables}

In this section we will describe basic constructions related to star products with separation of variables on a pseudo-K\"ahler manifold obtained in \cite{CMP1}, \cite{LMP1}, and \cite{LMP2}. Fix a pseudo-K\"ahler manifold $M$ with a pseudo-K\"ahler form $\omega_{-1}$ and an element
\[
    \omega = \frac{1}{\nu}\, \omega_{-1} + \omega_0 + \nu \omega_1 + \ldots
\]
of $\Omega(M)$. Let $U$ be a contractible coordinate chart on $M$. Each form $\omega_r, r \geq -1,$ has a local potential $\Phi_r$ on $U$ so that $\omega_r = i\p\bar\p \Phi_r$. Thus,
\[
                 \Phi = \frac{1}{\nu}\, \Phi_{-1} + \Phi_0 + \nu \Phi_1 + \ldots
\]
is a formal potential of $\omega$. The metric tensor $g_{k\bar l}$ of the form $\omega_{-1}$ is given on $U$ by the formula
\[
                   g_{k\bar l} = \frac{\p^2 \Phi_{-1}}{\p z^k \p \bar z^l}.
\]
Its inverse $g^{\bar lk}$ is a Poisson tensor of type (1,1). It was proved in \cite{CMP1} that there exists a unique globally defined star product with separation of variables $\ast_\omega$ on $(M, \omega_{-1})$ such that on each contractible coordinate chart $U$,
\[
                 L_{\frac{\p \Phi}{\p z^k}} = \frac{\p \Phi}{\p z^k} + \frac{\p}{\p z^k} \mbox{ and }  R_{\frac{\p \Phi}{\p \bar z^l}} = \frac{\p \Phi}{\p \bar z^l} + \frac{\p}{\p \bar z^l}.
\]
The mapping $\omega \mapsto \ast_\omega$ is a bijection of $\Omega(M)$ onto the set of all nondegenerate star products with separation of variables on $M$. The formal form $\omega$ is called the classifying form of the star product $\ast_\omega$. We drop the subscript $\omega$ in $\ast_\omega$ if it does not lead to confusion.

Let $(M,\omega_{-1})$ be a pseudo-K\"ahler manifold of complex dimension ~$m$ and $U \subset M$ be a contractible coordinate chart. Denote $\mathbf{g} = \det (g_{k\bar l})$ and fix a branch of $\log \mathbf{g}$ on~ $U$. The Ricci form of the metric $g_{k\bar l}$ is a closed global (1,1)-form on $M$ given locally by the formula $\rho = - i\p\bar\p \log\mathbf{g}$. We denote by $\varepsilon = [-\rho]$ the canonical class of $M$.

A local construction of the canonical trace density $\mu$ of a star product with separation of variables $\ast$ on $M$ was introduced in ~\cite{LMP2}.
Below we give a slightly modified version of this construction. Fix  an arbitrary formal potential
\begin{equation}\label{E:phi}
\Phi = \frac{1}{\nu}\, \Phi_{-1} + \Phi_0 + \nu \Phi_1 + \ldots
\end{equation}
of the classifying form $\omega$ of the product $\ast$ on $U$. There exists a unique potential $\Psi$ of the dual form $\tilde\omega$ of the form
\begin{equation}\label{E:psi}
     \Psi = - \frac{1}{\nu}\, \Phi_{-1} + (-\Phi_0 + \log \mathbf{g}) + \nu \Psi_1 + \ldots
\end{equation}
satisfying the equation
\begin{equation}\label{E:dnuphi}
       \frac{d\Phi}{d\nu} + I\frac{d\Psi}{d\nu} = \frac{m}{\nu},
\end{equation}
where $I$ is the formal Berezin transform for the product~ $\ast$. The global function $\varkappa$ from (\ref{E:mu}) is given by the formula
\begin{equation}\label{E:kappa}
      \varkappa = \Phi + \Psi - \log \mathbf{g} = \nu(\Phi_1 + \Psi_1) + \nu^2(\Phi_2 + \Psi_2) + \ldots
\end{equation}
on $U$. It follows from (\ref{E:kappa}) that
\begin{equation}\label{E:omegatildeomega}
             \omega + \tilde\omega = - \rho + i\p \bar\p \varkappa.
\end{equation}
Since $\varkappa$ is global, the class of $i\p \bar\p \varkappa$ is trivial and
\[
          [\omega] + [\tilde\omega] = \varepsilon. 
\]
It was proved in \cite{LMP1} that the class of the star product with separation of variables with the classifying form $\omega$ is $[\omega] - \varepsilon/2$, which is exactly the class of the corresponding phase form $\omega^{ph}$.

\section{A star product with a given phase form}

In this section we will prove the existence and uniqueness of a nondegenerate star product with separation of variables on a pseudo-K\"ahler manifold $M$ whose phase form is a given arbitrary element of $\Omega(M)$.

Let $M$ be a pseudo-K\"ahler manifold of complex dimension~ $m$ and~ $\ast$ be a nondegenerate star product with separation of variables on ~$M$ with the classifying form
\[
        \omega = \frac{1}{\nu}\, \omega_{-1} + \omega_0 + \nu \omega_1 + \ldots.
\]
\begin{lemma}\label{L:conomega}
   For each $r \geq 1$, the bidifferential operator $C_r$ of the star product $\ast$ depends only on the forms $\omega_k$ with $-1 \leq k \leq r-2$.
\end{lemma}
\begin{proof}
   Let $U \subset M$ be a contractible coordinate chart and $\Phi = \frac{1}{\nu}\Phi_{-1} + \Phi_0 + \ldots$ be a potential of $\omega$ on $U$. Fix $f \in C^\infty(U)$. It was proved in \cite{CMP1} that the formal differential operator $A = L_f$ on $U$ is completely determined by the following conditions: the operator $A$ does not contain antiholomorphic derivatives, it commutes with the operators $R_{\p\Phi/\p \bar z^l}$ for $1 \leq l \leq m,$ and satisfies $A1 = f$. In particular, $A = f + \nu A_1 + \nu^2 A_2 + \ldots$ and $A_r1 =0$ for all $r \geq 1$. For a function $g \in C^\infty(U), A_rg = C_r(f,g)$. The commutation condition written explicitly is as follows,
\begin{equation}\label{E:commut}
      \left[f + \nu A_1 + \ldots, \frac{1}{\nu}\frac{\p \Phi_{-1}}{\p \bar z^l} + \left(\frac{\p \Phi_0}{\p \bar z^l} + \frac{\p}{\p \bar z^l}\right) + \nu \frac{\p \Phi_1}{\p \bar z^l} + \ldots \right] =0.
\end{equation}
Observe that condition (\ref{E:commut}) does not depend on the choice of a formal potential of the form $\omega$. Extracting the component of (\ref{E:commut}) corresponding to $\nu^{r-1}$ we get
\begin{equation}\label{E:express}
     \left[A_r,  \frac{\p \Phi_{-1}}{\p \bar z^l}\right] + \left[A_{r-2}, \frac{\p \Phi_0}{\p \bar z^l} + \frac{\p}{\p \bar z^l}\right] + \ldots + \left[A_1, \frac{\p \Phi_{r-2}}{\p \bar z^l}\right] =0.
\end{equation}
Using induction on $r$, we see from (\ref{E:express}) that the commutator
\begin{equation}\label{E:commrs}
                  \left[A_r,  \frac{\p \Phi_{-1}}{\p \bar z^l}\right]
\end{equation}
is expressed in terms of the forms $\omega_k$ with $-1 \leq k \leq r-2$. It was shown in \cite{CMP1} that the knowledge of the commutators (\ref{E:commrs}) for $1 \leq l \leq m$ and the condition $A_r1=0$ determine $A_r$ uniquely. It follows that the operator $A_r$, and therefore $C_r$, are expressed  in terms of the forms $\omega_k$ with $-1 \leq k \leq r-2$.
\end{proof}
Let $I = 1 + \nu I_1 + \ldots$ be the formal Berezin transform of a nondegenerate star product with separation of variables $\ast$  on $M$ with the classifying form $\omega$. According to Lemma \ref{L:conomega}, for each $r \geq 1$ the operator $I_r$  is expressed  in terms of the forms $\omega_k$ with $-1 \leq k \leq r-2$. Set $\tilde I = I^{-1}$.  The operator $\tilde I = 1 + \nu \tilde I_1 + \nu^2 \tilde I_2 + \ldots$ is the formal Berezin transform of the dual star product $\tilde\ast$. The following lemma is trivial.
\begin{lemma}\label{L:inv}
   For each $r \geq 1$, the operator $\tilde I_r$ is expressed  in terms of the forms $\omega_k$ with $-1 \leq k \leq r-2$.
\end{lemma}
Assume that
\[
      \omega^{ph} = \frac{1}{\nu} \omega_{-1} + \omega^{ph}_0 + \nu \omega^{ph}_1 + \ldots
\]
is an arbitrary element of $\Omega(M)$. We want to  construct a star product with separation of variables $\ast$ on $(M,\omega_{-1})$ whose phase form is $\omega^{ph}$ and show its uniqueness. We will construct inductively its classifying form
\[
     \omega = \frac{1}{\nu} \omega_{-1} + \omega_0 + \nu \omega_1+ \ldots.
\]
It follows from formulas (\ref{E:phi}) and  (\ref{E:psi}) that
\[
       \omega_0 = \omega_0^{ph} - \frac{1}{2}\, \rho.
\] 
Let $U \subset M$ be a contractible coordinate chart. We rewrite equation~ (\ref{E:dnuphi}) on $U$ as follows,
\begin{equation}\label{E:commr}
                \frac{d\Psi}{d\nu} = \frac{m}{\nu} - \tilde I\,  \frac{d\Phi}{d\nu}.
\end{equation}
Extracting the component of (\ref{E:commr}) corresponding to $\nu^{r-1}$ for $r\geq 1$ we get the equation
\[
      r \Psi_r = - r\Phi_r - \sum_{k = 1}^{r+1} (r-k) \tilde I_k \Phi_{r-k},
\]
whence it follows that for $r \geq 1$,
\begin{equation}\label{E:omegar}
       \omega_r = \omega_r^{ph} - \frac{i}{2r} \p \bar \p \sum_{k = 1}^{r+1} (r-k) \tilde I_k \Phi_{r-k}.
\end{equation}
The sum on the right-hand side of (\ref{E:omegar}) is a global (1,1)-form on $M$  expressed  in terms of the forms $\omega_k$ with $-1 \leq k \leq r-1$. Therefore, the form $\omega$ can be inductively constructed from the phase form $\omega^{ph}$ and is uniquely defined.
We have proved the following theorem.
\begin{theorem}\label{T:existuniq}
  Given a pseudo-K\"ahler manifold $M$, for any formal form $\omega^{ph} \in \Omega(M)$ there exists a unique nondegenerate star product with separation of variables on $M$ whose phase form is $\omega^{ph}$.
\end{theorem}

It is particularly easy to construct a classifying form $\omega$ with a given phase form $\omega^{ph}$ if the phase form is an invariant formal form on a homogeneous pseudo-K\"ahler manifold. A nondegenerate star product with separation of variables on a homogeneous pseudo-K\"ahler manifold is invariant if and only if its classifying form is invariant (see \cite{MBN}). If $\ast$ is an invariant star product with separation of variables on a homogeneous pseudo-K\"ahler manifold $M$, then its canonical trace density $\mu$ is invariant and therefore the function $\varkappa$ from (\ref{E:mu}) is a formal constant. Now formula (\ref{E:omegatildeomega}) implies that
\[
                  \omega + \tilde \omega = - \rho
\]
and the corresponding phase form is
\[
            \omega^{ph} = \omega + \frac{1}{2}\, \rho. 
\]
Vice versa, if $\omega^{ph}$ is an arbitrary invariant formal form from $\Omega(M)$, then
\[
     \omega = \omega^{ph} - \frac{1}{2}\, \rho 
\] 
is an invariant classifying form of an invariant star product with separation of variables on $M$ whose corresponding phase form is $\omega^{ph}$.

{\it Example:} The complex projective space $\C P^m$ equipped with the Fubini-Study form $\omega_{FS}$ is a homogeneous K\"ahler manifold under the action of the projective unitary group $PU(m+1)$. The Ricci form of the Fubini-Study metric is $(m+1)\omega_{FS}$. 
The invariant star product with the classifying form
\[
     \omega = \frac{1}{\nu}\, \omega_{FS} - \frac{m+1}{2}\omega_{FS}
\]
has the phase form
\[
     \omega^{ph} = \frac{1}{\nu}\, \omega_{FS}.
\]

\section{Change of the formal parameter}

Let $\tau(\nu) = \tau_1 \nu + \tau_2 \nu^2 + \ldots$ be a formal series in $\nu$ with $\tau_r \in \C$ and $\tau_1 \neq 0$. The change of the formal parameter $\nu \mapsto \tau(\nu)$ defines, via a pullback, a $\C$-algebra automorphism $T = \tau^\ast$ of $C^\infty(M)[\nu^{-1},\nu]]$,
\[
      (Tf)(\nu,x) = f(\tau(\nu), x),\ x \in M.
\]
The action of $T$ extends to other formal geometric objects on $M$.
We will be particularly interested in the involutive change of the formal parameter $\nu \mapsto -\nu$.

Given a star product $\ast$ on $(M, \{\cdot, \cdot\})$, the product $\ast_T$ defined by
\[
      f \ast_T g = T( (T^{-1}f) \ast (T^{-1}g))
\]
is a $\C[\nu^{-1},\nu]]$-bilinear star product on $(M, \tau_1\{\cdot, \cdot\})$. If the product $\ast$ is given by a formal bidifferential operator $C = \sum_{r =0}^\infty \nu^r C_r$, then the product $\ast_T$ is given by 
\[
C_T = \sum_{r =0}^\infty (\tau(\nu))^r C_r.
\]
\begin{lemma}\label{L:it}
If $\ast$ is a (possibly degenerate) star product with separation of variables on a complex manifold $M$ with the formal Berezin transform $I$, then $\ast_T$ is also a star product with separation of variables whose formal Berezin transform $I_T$ is given by the formula $I_T = T I T^{-1}$. 
\end{lemma}
\begin{proof}
   If $a$ is a local holomorphic function on $M$, then so is $T^{-1}a$. We have
\[
    a \ast_T f =  T( (T^{-1}a) \ast (T^{-1}f)) =  T( (T^{-1}a)(T^{-1}f)) = af.
\] 
Similarly, for a local antiholomorphic function $b$ we have $f \ast_T b = bf$. Therefore, $\ast_T$ is a star product with separation of variables. Now,
\begin{eqnarray*}
      I_T(ab) = b \ast_T a = T( (T^{-1}b) \ast (T^{-1}a)) =\\
 TI((T^{-1}a) (T^{-1}b)) = TIT^{-1}(ab),
\end{eqnarray*}
hence $I_T = TIT^{-1}$.
\end{proof}

In the rest of the paper we will assume that $\ast$ is a nondegenerate star product with separation of variables on a pseudo-K\"ahler manifold $(M, \omega_{-1})$ of complex dimension $m$ with the classifying form
\[
             \omega = \frac{1}{\nu}\omega_{-1} + \omega_0 + \ldots \in \Omega(M).
\]
We fix a change of the formal variable $T = \tau^\ast$. The product $\ast_T$ is a star product with separation of variables on the pseudo-K\"ahler manifold $(M, (1/\tau_1)\omega_{-1})$.
\begin{lemma}\label{L:omega}
 The classifying form of the star product with separation of variables $\ast_T$ is $T\omega$.
\end{lemma}
\begin{proof}
Let $\Phi$ be a formal potential of $\omega$ on a contractible coordinate chart $U \subset M$. Given a function $f$ on $U$, we have
\begin{eqnarray*}
\frac{\p (T\Phi)}{\p z^k} \ast_T f =
T\left(\frac{\p \Phi}{\p z^k}\right) \ast_T f = T\left(\frac{\p \Phi}{\p z^k} \ast T^{-1}f\right) = \\
 T\left(\frac{\p \Phi}{\p z^k} T^{-1}f + \frac{\p}{\p z^k} T^{-1}f \right) = \frac{\p (T\Phi)}{\p z^k}  f + \frac{\p f}{\p z^k}.
\end{eqnarray*}
It follows that $T\Phi$ is a formal potential which determines the star product with separation of variables $\ast_T$ on $U$. Therefore, the classifying form of $\ast_T$ is $T\omega$.
\end{proof}
Let $I$ be the formal Berezin transform and $\mu$ be the canonical trace density for the product $\ast$,  and let $\tilde\omega$ be the dual of the form $\omega$. 
The canonical trace density of the star product $\ast_T$ is given by the formula
\begin{equation}\label{E:mut}
      \mu_T = \frac{1}{m!}\left(\frac{1}{\tau_1 \nu}\omega_{-1}\right)^m e^{\varkappa_T},
\end{equation}
where $\varkappa_T$ is a globally defined formal function on $M$ which vanishes at $\nu =0$. The expression
\[
          \log\frac{\tau(\nu)}{\tau_1\nu} = \frac{\tau_2}{\tau_1}\nu + \ldots
\]
gives a well defined formal series which also vanishes at $\nu =0$.
\begin{proposition}\label{P:tildeomega}
\begin{enumerate}[(a)]
 \item The dual form of the form $T\omega$ is the form $T\tilde\omega$. 
\item The following formula holds,
\[
     \varkappa_T = T\varkappa - m \log \frac{\tau(\nu)}{\tau_1 \nu}.
\]
\item The canonical trace density for the product $\ast_T$ is
\begin{equation}\label{E:tmu}
          \mu_T = T\mu = \frac{1}{m!}\left(\frac{1}{\tau(\nu)}\, \omega_{-1}\right)^m e^{T\varkappa}. 
\end{equation}
\end{enumerate}
\end{proposition}
\begin{proof}
Let $U \subset M$ be a contractible coordinate chart. Denote, as above, the metric tensor of $\omega_{-1}$ by $g_{k\bar l}$ and set
$\mathbf{g} = \det (g_{k\bar l})$. The metric tensor for $(1/\tau_1)\omega_{-1}$ is $(1/\tau_1)g_{k\bar l}$ and 
\[
    \det \left(\frac{1}{\tau_1}g_{k\bar l}\right) = \left(\frac{1}{\tau_1}\right)^m \mathbf{g}.
\]
Fix a branch of $\log \mathbf{g}$ on $U$ and a value of $\log \tau_1$, and set
\[
              \log \left(\frac{1}{\tau_1}\right)^m \mathbf{g} = \log \mathbf{g} - m \log \tau_1.
\]
Choose an arbitrary potential $\Phi$ of the form $\omega$ on $U$. Then, according to Lemma \ref{L:omega}, $T\Phi$ is a potential of the form $\omega_T$. We have
\[
      T\Phi = \frac{1}{\tau_1 \nu}\Phi_{-1} + \left(-\frac{\tau_2}{\tau_1^2}\Phi_{-1}+\Phi_0\right) \pmod{\nu}.
\]
 There exists a unique potential $\Xi$ of the form dual to $T\omega$ such that
\begin{equation}\label{E:longxi}
    \Xi = - \frac{1}{\tau_1\nu}\Phi_{-1} + \left(\frac{\tau_2}{\tau_1^2}\Phi_{-1}-\Phi_0 + \log \mathbf{g} - m \log \tau_1\right) \pmod{\nu},
\end{equation}
and which satisfies the equation
\begin{equation}\label{E:itxi}
          \frac{d}{d\nu}(T\Phi)  +I_T \frac{d}{d\nu}\Xi = \frac{m}{\nu}.
\end{equation}
The function $\varkappa_T$ in formula (\ref{E:mut}) has the following local expression,
\begin{equation}\label{E:kappatau}
      \varkappa_T = T\Phi + \Xi -  \log \mathbf{g} + m \log \tau_1.
\end{equation}
Applying $T$ to (\ref{E:dnuphi}) and using the fact that
\[
       T \frac{d}{d\nu}T^{-1} = \frac{1}{\tau'(\nu)}\frac{d}{d\nu},
\]
we obtain that
\begin{equation}\label{E:overtau}
     \frac{1}{\tau'(\nu)}\frac{d (T\Phi)}{d\nu} + \frac{1}{\tau'(\nu)} I_T \frac{d (T\Psi)}{d\nu} = \frac{m}{\tau(\nu)}.
\end{equation}
Equation (\ref{E:overtau}) is equivalent to the following one,
\begin{equation}\label{E:logfrac}
     \frac{d}{d\nu} (T\Phi) + I_T \frac{d}{d\nu}\left(T\Psi - m\log\frac{\tau(\nu)}{\tau_1\nu}\right) = \frac{m}{\nu}.
\end{equation}
We get from equations (\ref{E:itxi}) and (\ref{E:logfrac}) that
\begin{equation}\label{E:xitpsi}
                      \frac{d}{d\nu}\Xi = \frac{d}{d\nu}\left(T\Psi - m\log\frac{\tau(\nu)}{\tau_1\nu}\right).
\end{equation}
Equation (\ref{E:psi}) implies that
\begin{equation}\label{E:longtpsi}
      T\Psi = - \frac{1}{\tau_1\nu}\Phi_{-1} + \left(\frac{\tau_2}{\tau_1^2}\Phi_{-1} - \Phi_0 + \log \mathbf{g}\right) \pmod{\nu}.
\end{equation}
We obtain from Eqns. (\ref{E:longxi}, \ref{E:xitpsi}, \ref{E:longtpsi}) that
\begin{equation}\label{E:xiviatpsi}
     \Xi = T\Psi - m \log \tau_1 - m\log\frac{\tau(\nu)}{\tau_1\nu}.
\end{equation}
Statement $(a)$ of  the Proposition follows directly from formula (\ref{E:xiviatpsi}). We obtain statement $(b)$ combining the formula
\[
     T\varkappa = T\Phi + T\Psi - \log \mathbf{g}
\]
with (\ref{E:kappatau}) and (\ref{E:xiviatpsi}). Statement $(b)$ readily implies $(c)$.
\end{proof}

Let $\omega^{ph}$ denote the phase form of the product $\ast$.
\begin{corollary}
   The phase form of the star product $\ast_T$ is $T\omega^{ph}$.
\end{corollary}

Now assume that $\tau(\nu) = \tau_1\nu  + \tau_2 \nu^2 + \ldots$ is a proper involution, i.e., $\tau \circ \tau = id$ and $\tau_1 = -1$, for example, $\tau(\nu) = - \nu$. For any $\omega \in \Omega(M)$, its odd part with respect to the proper involution $T = \tau^\ast$,
\[
        \frac{1}{2}(\omega - T\omega),
\]
also lies in $\Omega(M)$. 

\begin{theorem}\label{T:equiv}
   Let $\ast$ be a nondegenerate star product with separation of variables on a pseudo-K\"ahler manifold $M$ with the classifying form $\omega \in \Omega(M)$, dual form $\tilde\omega$, phase form $\omega^{ph}$, and formal Berezin transform $I$, and let $T=\tau^\ast$ be a proper involution. Then the following conditions are equivalent:
\begin{enumerate}[(i)]
\item $T\omega^{ph} = - \omega^{ph}$;
\item $T\omega = \tilde \omega$;
\item $I_T = I^{-1}$.
\end{enumerate}
\end{theorem}
\begin{proof}
$(i)\Rightarrow (ii)$.  Assume that $\omega^{ph}$ is odd with respect to $T$, i.e., $T\omega^{ph} = - \omega^{ph}$.  The phase form of the form $T\omega$ is $T\omega^{ph} = - \omega^{ph}$. The phase form of the dual form $\tilde\omega$ is also $-\omega^{ph}$. By Theorem \ref{T:existuniq}, $T\omega = \tilde\omega$. The implications $(ii)\Rightarrow (i)$ and $ (ii) \Leftrightarrow (iii)$ are straightforward.
\end{proof}

Assume that $T = \tau^\ast$ is a proper involution and $\ast$ is a star product satisfying conditions $(i) - (iii)$ of Theorem \ref{T:equiv}.
\begin{proposition}\label{P:kappat}
 The canonical trace density $\mu_T$ of the star product $\ast_T$ is expressed through the canonical trace density $\mu$ of the product $\ast$ as follows,
\[
              \mu_T =   (-1)^m \mu. 
\]
\end{proposition}
\begin{proof}
Formula (\ref{E:mut}) and the fact that $\tau_1 = -1$ imply that the statement of the Proposition is equivalent to the statement
\begin{equation}\label{E:ident}
            \varkappa_T = \varkappa.
\end{equation}
We will prove the formula
\begin{equation}\label{E:tkappaminus}
       T\varkappa - \varkappa = m \log \frac{\tau(\nu)}{\tau_1 \nu}
\end{equation}
which, according to part $(ii)$ of Proposition \ref{P:tildeomega}, is equivalent to (\ref{E:ident}).
Assume that $\omega \in \Omega(M)$ is such that $T\omega = \tilde\omega$. Let $\Phi$ and $\Psi$  be formal potentials of $\omega$ and $\tilde\omega$ on a contractible coordinate chart $U \subset M$ given by formulas (\ref{E:phi}) and (\ref{E:psi}), respectively. Formula (\ref{E:kappa}) implies that
\begin{equation}\label{E:tphitpsi}
       T\varkappa - \varkappa = T\Phi + T\Psi - \Phi - \Psi.
\end{equation}
Since $I_T = I^{-1}$ by Theorem \ref{T:equiv}, equation  (\ref{E:logfrac}) can be rewritten as follows,
\begin{equation}\label{E:invlogfrac}
     I\frac{d (T\Phi)}{d\nu} + \frac{d}{d\nu}\left(T\Psi - m\log\frac{\tau(\nu)}{\tau_1\nu}\right) = \frac{m}{\nu}.
\end{equation}
Subtracting equation (\ref{E:dnuphi}) from (\ref{E:invlogfrac}) we obtain that
\begin{equation}\label{E:subtr}
  I \frac{d}{d\nu}(T\Phi - \Psi) + \frac{d}{d\nu}\left(T\Psi - \Phi - m\log\frac{\tau(\nu)}{\tau_1\nu}\right) = 0.
\end{equation}
 Lemma \ref{L:omega} implies that $T\Phi - \Psi = a + b$, where $a$ and $b$ are a formal holomorphic and antiholomorphic functions on $U$, respectively, whence
\[
      I \frac{d}{d\nu}(T\Phi - \Psi) =  \frac{d}{d\nu}(T\Phi - \Psi).
\]
We get from (\ref{E:subtr}) that
\begin{equation}\label{E:almost}
      \frac{d}{d\nu}\left(T\Phi - \Psi + T\Psi - \Phi - m\log\frac{\tau(\nu)}{\tau_1\nu}\right) =0.
\end{equation}
We see from equations (\ref{E:tphitpsi}) and (\ref{E:almost}) that
\begin{equation}\label{E:follows}
          \frac{d}{d\nu}\left(T\varkappa - \varkappa - m\log\frac{\tau(\nu)}{\tau_1\nu}\right) =0.
\end{equation}
Formula (\ref{E:tkappaminus}) follows from (\ref{E:follows}) and the fact that the formal series $T\varkappa, \varkappa$, and $\log(\tau(\nu)/\tau_1\nu)$  vanish at $\nu =0$.  
\end{proof}

\begin{corollary}
   If $\tau(\nu) = -\nu$, then the function $\varkappa$ is even in the formal parameter $\nu$.
\end{corollary}
\begin{proof}
The Corollary is an immediate consequence of formula  (\ref{E:tkappaminus}).
\end{proof}

Given the formal Berezin transform $I$ corresponding to a classifying form $\omega = (1/\nu)\omega_{-1} + \ldots \in \Omega(M)$, the formal differential operator $X := \log I = \nu X_1 + \nu^2 X_2 + \ldots$ is a well defined global operator on $M$ with $X_1 = \Delta$  the Laplace-Beltrami operator for the pseudo-K\"ahler metric $\omega_{-1}$. Condition $(iii)$ of Theorem ~\ref{T:equiv} is equivalent to the condition that $X$ is odd with respect to $T$, i.e., $TXT^{-1} = - X$. In \cite{CMP3} it was noticed that, for each $r \geq 1$, the order of the operators $X_{2r-1}$ and $X_{2r}$ is not greater than $2r$. This observation leads to the question whether the order of the operators $X_{2r}$ can be lowered further. It turns out that all operators $X_{2r}$ can simultaneously vanish. Namely, if the involution is $\tau(\nu) = -\nu$, one can see from Theorem \ref{T:equiv} that all operators $X_{2r}$ vanish if and only if the phase form $\omega^{ph}$ of the form $\omega$ is odd in the formal parameter $\nu$.

\end{document}